\documentclass[12pt]{article}
\usepackage{amsfonts}
\usepackage{amsmath}
\usepackage{amssymb}
\usepackage{amsthm}
\usepackage{geometry}
\usepackage{rotating}
\usepackage{multirow}
\usepackage[utf8]{inputenc}
\usepackage{array}
\usepackage[english]{babel}
\usepackage{float}

\usepackage{bm}
\usepackage{natbib}

\usepackage{setspace}

\theoremstyle{plain} \newtheorem{lemma}{Lemma}
\theoremstyle{plain} \newtheorem{theorem}{Theorem}
\theoremstyle{plain} \newtheorem{corollary}{Corollary}

\begin{document}
\title{The cost of using exact confidence intervals for a binomial proportion}

\author
{ {M{\aa}ns Thulin}\\{\footnotesize{Department of Mathematics, Uppsala University}}}
\date{}

\maketitle


\begin{abstract}
\noindent 
When computing a confidence interval for a binomial proportion $p$ one must choose between using an exact interval, which has a coverage probability of at least $1-\alpha$ for all values of $p$, and a shorter approximate interval, which may have lower coverage for some $p$ but that on average has coverage equal to $1-\alpha$. We investigate the cost of using the exact one and two-sided Clopper--Pearson confidence intervals rather than shorter approximate intervals, first in terms of increased expected length and then in terms of the increase in sample size required to obtain a desired expected length. Using asymptotic expansions, we also give a closed-form formula for determining the sample size for the exact Clopper--Pearson methods. For two-sided intervals, our investigation reveals an interesting connection between the frequentist Clopper--Pearson interval and Bayesian intervals based on noninformative priors.
\noindent 
   \\[1.5mm] {\bf Keywords:} Asymptotic expansion; binomial distribution; confidence interval; expected length; sample size determination; proportion.
\end{abstract}

\section{Introduction}\label{introduction}
Inference for a binomial proportion $p$ is one of the most commonly encountered statistical problems, with important applications in areas such as clinical trials, risk analysis and quality control. Consequently, a large number of two-sided confidence intervals and one-sided confidence bounds for $p$ have been proposed by different authors. These are of two different types: \emph{exact} methods, that have a coverage at least equal to $1-\alpha$ for all $p\in(0,1)$, and \emph{approximate} methods, that may have coverage less than $1-\alpha$ for some values of $p$, but that have a coverage that in some sense is approximately equal to $1-\alpha$.

Research on confidence intervals and bounds for a binomial proportion has mostly focused on approximate intervals. In the methodological literature, exact intervals have often been deemed to be too conservative \citep{ac1,bcd1,nn1}, as they tend to be quite wide and have actual coverage levels that often are noticeably greater than $1-\alpha$. Nevertheless, the use of exact intervals for proportions is abundant among practitioners: see e.g. \citet{ab1}, \citet{ib1}, \citet{wa2} and \citet{su1} for some recent examples. By far the most widely used exact interval is the Clopper--Pearson interval, introduced by \citet{cp1}.

The benefit of using an exact interval is obvious: one does not risk that the actual coverage falls below $1-\alpha$. For this reason, some regulatory authorities require that exact intervals be used. Moreover, the binomial distribution is unusual in that we often can be sure that it is an accurate description of that which we are modelling and not just an approximation to the true distribution, as is often the case when continuous distributions are used for modelling. In such a situation, using an exact method seems reasonable. But there are also costs associated with the use of such an interval. When choosing between approximate and exact confidence methods, there is a trade-off in that exact intervals and bounds by construction are wider than the best approximate intervals, or equivalently, require a larger sample size in order to obtain a certain expected length. If one is unwilling to accept intervals and bounds with undercoverage for some values of $p$, there is a cost to pay in terms of expected length or required sample size. This paper seeks to quantify these costs.

In planned experiments, it is always important to determine a suitable sample size. Sample size determination for binomial confidence intervals has received much attention in recent years \citep{ka1,pi1,kp1,la1,go1,we1}, with different authors studying different intervals and methods for sample size calculations, the latter often of a computer-intensive nature. The first main contribution of this paper is closed-form formulas for computing the sample size required for the Clopper--Pearson methods to obtain a given expected length. This eliminates the need for computer-intensive methods for computing sample sizes and gives a better understanding of how the desired length and the parameters $p$ and $\alpha$ affect the sample size.


The second main contribution is closed-form expressions for the excess length and increase in required sample size that comes from using the exact Clopper--Pearson methods instead of approximate methods. We obtain these expressions by deriving asymptotic expansions for the exact Clopper--Pearson methods, extending the work of \citet{bcd2}, \citet{cai1} and \citet{st1} on the asymptotics of approximate binomial confidence methods to exact intervals and bounds.

The rest of the paper is organised as follows. In Section \ref{methods} we introduce the Clopper--Pearson methods along with other exact and approximate confidence methods. In Section \ref{twosided}  we give an asymptotic expression for the expected length of the Clopper--Pearson interval. This allows us to give a formula for computing the sample size, and to determine the cost of using an exact interval rather than an approximate interval, in terms of expected length and sample size. In Section \ref{onesided} we discuss the one-sided Clopper--Pearson bound and give expressions for its expected distance to $p$ and the cost of using an exact bound. In Section \ref{disc} we discuss costs associated with approximate intervals and state some conclusions. All proofs and technical details are deferred to an appendix.

\section{Binomial confidence methods}\label{methods}


\subsection{The Clopper--Pearson interval and bounds}
The two-sided Clopper--Pearson interval  for a proportion $p$ is an inversion of the equal-tailed binomial test: the interval contains all values of $p$ that aren't rejected by the test at confidence level $\alpha$. Given an observation $X$, the lower limit is thus given by the value of $p_{L}$ such that
\begin{equation}\label{cp1}
\sum_{k=X}^n\binom{n}{k}p_{L}^k(1-p_{L})^{n-k}=\alpha/2
\end{equation}
and the upper limit is given by the $p_{U}$ such that
\begin{equation}\label{cp2}
\sum_{k=0}^{X}\binom{n}{k}p_{U}^k(1-p_{U})^{n-k}=\alpha/2.
\end{equation}


As is well-known, the computation of $p_{L}$ and $p_{U}$ is simplified by the following equality from \citet{kotz1}. Let $f(t,a,b)$ be the density function of a $Beta(a,b)$ random variable. Then
\begin{equation}\label{kotz}
\sum_{k=X}^n\binom{n}{k}p^k(1-p)^{n-k}=\int_0^pf(t,X,n-X+1)dt.
\end{equation}
When (\ref{kotz}) is plugged into (\ref{cp1}) and (\ref{cp2}), the problem of finding $p_{L}$ and $p_{U}$ reduces to inverting the distribution functions of two beta distributions. Consequently, the endpoints of the Clopper--Pearson interval are given by quantiles of beta distributions:
\begin{equation}\label{betaapprox}
(p_{L},p_{U})=\Big{(}B(\alpha/2,X,n-X+1),\quad B(1-\alpha/2,X+1,n-X)\Big{)}.
\end{equation}
When $X$ is either 0 or $n$, closed-form expressions for the interval bounds are available. When $X=0$ the interval is $(0,1-(\alpha/2)^{1/n})$ and when $X=n$ it is $((\alpha/2)^{1/n},1)$. For other values of $X$, (\ref{betaapprox}) must be evaluated numerically. The interval is implemented in most statistical software packages; it can for instance be found in the \texttt{PropCIs} package in R and computed using the \texttt{PROC FREQ} command in SAS. 

Some authors \citep{ac1,bcd1} have argued that when choosing between confidence intervals, it is often preferable to use an interval with a simple closed-form formula rather than one that requires numerical evaluation, as the former is easier to present and to interpret. Next, we give asymptotic expansions of $p_{L}$ and $p_{U}$, that function as good approximations when $n\geq 40$, and can be used if a closed-form formula for the Clopper--Pearson interval is desired. As an example, when $n=50$ the upper bound is accurate up to two decimal places for $X\notin\{0,1,2,n\}$.

\begin{theorem}\label{boundapprox1}
Let $X\in\{1,2,\ldots,n-1\}$ be fixed. Let $\hat{p}=X/n$, $\hat{q}=1-\hat{p}$ and $z_{\alpha/2}$ be the upper $\alpha/2$ quantile of the standard normal distribution.

The bounds of the Clopper--Pearson interval are, up to $O(n^{-3/2})$,
\[\begin{split}
p_L=\hat{p}&-n^{-1/2}z_{\alpha/2}(\hat{p}\hat{q})^{1/2}+(3n)^{-1}\Big{(}2(1/2-\hat{p})z_{\alpha/2}^2-
(1+\hat{p})\Big{)}\qquad\mbox{and}\\
p_U=\hat{p}&+n^{-1/2}z_{\alpha/2}(\hat{p}\hat{q})^{1/2}+(3n)^{-1}\Big{(}2(1/2-\hat{p})z_{\alpha/2}^2+1+
\hat{q}\Big{)}.
\end{split}\]
\end{theorem}
\noindent
Similar in construction to the two-sided interval, the one-sided Clopper--Pearson bounds are obtained by inverting one-sided binomial tests. Thus the $1-\alpha$ Clopper--Pearson upper bound $p_U$ is given by the $p_U$ such that
\begin{equation}
\sum_{k=0}^{X}\binom{n}{k}p_{U}^k(1-p_{U})^{n-k}=\alpha.
\end{equation}
In the following, we limit our study to upper bounds. For symmetry reasons, the results are however equally valid for lower bounds, as for the bounds under consideration, a lower bound $p_L$ for $p$ is equivalent to an upper bound for $q$, as $q_U=1-p_L$.

If a closed-form expression for $p_U$ is desired, it can be obtained in the form of an asymptotic expansion by replacing $\alpha/2$ with $\alpha$ in Theorem \ref{boundapprox1} above.


\subsection{Other exact intervals}
In much of the medical literature, as well as the rest of the present paper, the Clopper--Pearson interval is refered to as \emph{the} exact confidence interval for a binomial proportion. Despite this terminology, several other exact intervals have been proposed throughout the years. These alternative intervals do not admit closed-form expressions and are, to varying extents, computer-intensive.

There are several reasons as to why the Clopper--Pearson interval is the most widely used exact interval. One is simply tradition and availability: it has found its way in to classic statistical textbooks and has been implemented in almost all statistical software packages. Compared to the computer-intensive alternatives, the Clopper--Pearson interval is also considerably simpler computationally. Finally, it remains a natural choice in that it is the inversion of the well-known equal-tailed binomial test. 

In the two-sided case, however, there is room for improvement, at least if one is willing to let go of some natural properties of confidence intervals. Other exact intervals have been designed to be shorter than the Clopper--Pearson interval, by inverting two-sided tests that need not be equal-tailed. Moreover, the coverage probabilities of these intervals often fluctuate less from $1-\alpha$ than does the coverage of the Clopper--Pearson interval. 

The Blyth--Still--Casella interval \citep{bs1,ca1} is guaranteed to be the shortest exact interval, but has the odd property that it is not nested, in the sense that the 90 \% interval need not be contained in the 95 \% interval \citep[Theorem 2]{bl1}. This is also true for the intervals of \citet{st2}.

The \citet{st3} procedure yields nested intervals that are shorter than the Clopper--Pearson interval, but will in some cases result in two separate intervals rather than one connected interval. \citet{bl1} proposed a nested exact interval that, while wider than the Blyth--Still--Casella interval, always is contained in the Clopper--Pearson interval. It is however sometimes a union of disjoint intervals and its upper bound is decreasing but not strictly decreasing in $\alpha$ when $n$ and $X$ are fixed \citep{vh1}. The interval based on the inverted exact likelihood ratio test suffers from similar problems \citep{vh1}.

The Clopper--Pearson interval, in contrast, is nested, is always a connected set and has bounds that are strictly monotone in $\alpha$. While it is possible to obtain shorter exact confidence intervals for a binomial proportion, this seems to be associated with the loss of nestedness, connectedness or monotonicity. As we consider these properties to be of importance, we will only include the Clopper--Pearson interval and bounds in the following sections, and will out of convenience refer to them as \emph{the} exact methods.

Implementations of some of the alternative exact intervals are readily available. The Blyth--Still--Casella interval has been implemented in StatXact and \citet{bl1} gave an S-PLUS function for his interval. A more efficient implementation of Blaker's interval is found in the R package \texttt{BlakerCI} \citep{kl1}.

\subsection{Approximate confidence intervals and bounds}\label{appendix1}
Throughout the text, the Clopper--Pearson methods will be compared to several well-known approximate methods. These are described below, along with the commonly used Wald interval. For more thorough reviews of binomial confidence methods, see \citet{ne1}, \citet{cai1} and \citet{bcd1,bcd2}. In the descriptions below, $\hat{p}=X/n$ is the sample proportion, $\hat{q}=1-\hat{p}$ and $z_{\alpha/2}$ is the $100(1-\alpha/2)$th percentile of the standard normal distribution.

\emph{The Wald interval.} Inversion of the large sample test $|(\hat{p}-p)(\hat{p}\hat{q}/n)^{-1/2}|\leq z_{\alpha/2}$ leads to the Wald interval, which is presented in virtually every introductory statistics course: $\hat{p}\pm z_{\alpha/2}\sqrt{\hat{p}\hat{q}/n}$. The Wald interval suffers from particularly erratic coverage properties, and cannot be recommended for general use \citep{bcd1,ne1}.

\emph{The Wilson score interval.} Like the Wald interval, the \citet{wi1} score interval is based on an inversion of the large sample normal test $|(\hat{p}-p)/d(\hat{p})|\leq z_{\alpha/2}$, where $d(\hat{p})$ is the standard error of $\hat{p}$. Unlike the Wald interval, however, the inversion is obtained using the null standard error $(p(1-p)/n)^{1/2}$ instead of the sample standard error. The solution of the resulting quadratic equation leads to the confidence interval
\[
\frac{X+z_{\alpha/2}^2/2}{n+z_{\alpha/2}^2}\pm \frac{z_{\alpha/2}}{n+z_{\alpha/2}^2}\sqrt{\hat{p}\hat{q}n+z_{\alpha/2}^2/4}.
\]
The Wilson score interval has favourable coverage and length properties and is often recommended for general use \citep{bcd1,ne1}.

\emph{The Agresti--Coull interval.} For $95 \%$ nominal coverage, \citet{ac1} proposed the use of the Wald interval with two successes and two failures added, i.e. with $n$ replaced by $n+4$ and $X$ replaced by $X+2$. More generally, let $\tilde{n}=n+z_{\alpha/2}^2$, $\tilde{X}=X+z_{\alpha/2}^2/2$, $\tilde{p}=\tilde{X}/\tilde{n}$ and $\tilde{q}=1-\tilde{p}$. \citet{bcd1} dubbed the interval
$
\tilde{p}\pm z_{\alpha/2}\sqrt{\tilde{p}\tilde{q}/\tilde{n}}
$
the Agresti-Coull interval. It has performance close to that of the Wilson interval, but is somewhat simpler to use.

\emph{Bayesian Beta intervals and bounds}. Let $B(\alpha,a,b)$ denote the $\alpha$-quantile of the $Beta(a,b)$ distribution. An equal-tailed Bayesian credible interval based on the $Beta(a,b)$ prior is given by $(B(\alpha/2,X+a,n-X+b),~ B(1-\alpha/2,X+a,n-X+b))$, where $B(\alpha,a,b)$ is the quantile function of the $Beta(a,b)$ distribution. Similarly, an upper bound is given by $B(1-\alpha,X+a,n-X+b)$. As these methods make use of beta quantiles, they are algebraically very similar to the Clopper--Pearson interval. This connection is discussed further in Section \ref{freqbayes}.

\emph{The Jeffreys interval and bound.} A commonly used Bayesian interval for $p$ is the Jeffreys interval $(B(\alpha/2,X+1/2,n-X+1/2),~ B(1-\alpha/2,X+1/2,n-X+1/2))$, which is the equal-tailed credible interval derived using the noninformative Jeffreys prior. Both the two-sided interval and the one-sided bound exhibit favourable frequentist properties \citep{bcd1,ne1,cai1}.


\emph{The second-order correct bound.} \citet{cai1} proposed a coverage-corrected version of the one-sided Wald bound, based on second-order asymptotic expansions. \citet{cai1} recommended it for general use and gave a closed-form expression for the bound.

\emph{The modified loglikelihood root bound.} \citet{st1} studied the bound obtained by inverting the modified loglikelihood root test and found it to have very favourable coverage and length properties. It cannot be expressed in a closed form, but \citet{st1} gave asymptotic expansions that can be used as approximations.

\section{Two-sided intervals}\label{twosided}


\subsection{Expected length}
Let $q=1-p$ and let $L_{CP}=p_U-p_L$ denote the length of the Clopper--Pearson interval. Next, we present an asymptotic expression for the expectation of $L_{CP}$.
\noindent 
\begin{theorem}\label{lenthm}
As $n\rightarrow\infty$ the expected length of the $1-\alpha$ Clopper--Pearson interval is
\begin{equation}\label{lenexp}\begin{split}
E(L_{CP})=&2z_{\alpha/2}n^{-1/2}(pq)^{1/2}+n^{-1}\\
&\qquad+n^{-3/2}(pq)^{-1/2}\frac{z_{\alpha/2}}{18}\Big{(}z_{\alpha/2}^2-\frac{5}{2}-17pq-13pqz_{\alpha/2}^2\Big{)}+O(n^{-2}).
\end{split}\end{equation}
\end{theorem}
The expansion (\ref{lenexp}) is compared to the actual expected length in Figure \ref{len1}. Even for small values of $n$, the approximation comes quite close to the actual expected length over the entire parameter space.

\begin{figure}[H]
\begin{center}
   \caption{Comparison between the actual expected length and the expansion (\ref{lenexp}) for the nominal 95 \% Clopper--Pearson interval.}\label{len1}
   \includegraphics[width=\textwidth]{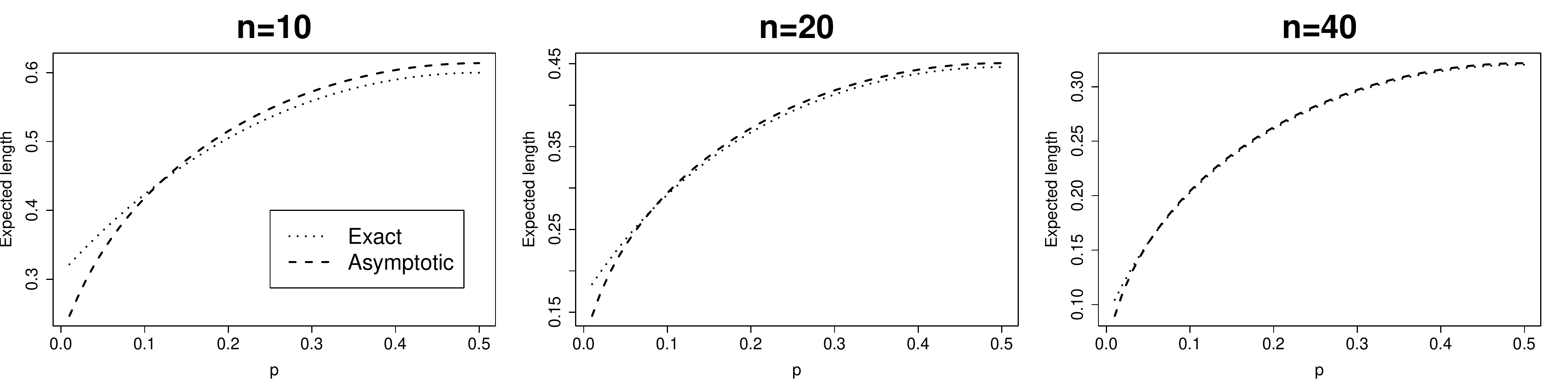}
\end{center}
\end{figure}

Having an expression for the expected length of the Clopper--Pearson interval allows us to evaluate its performance for different combinations of $n$, $p$ and $\alpha$. When planning an experiment, this is extremely useful as it can be used to determine what sample size we need in order to achieve a desired expected length. Methods for determining sample size are discussed next.


\subsection{Sample size determination}
Several different criterions can be considered when determining sample size, as discussed e.g. by \citet{go1}. We focus on a comparatively simple criterion: for a fixed confidence level $1-\alpha$ we wish to find the smallest sample size $n$ such that the expected length of the confidence interval is some fixed value $d$. As the value of $n$ will depend on $p$, we require that an initial guess $p_0$ for $p$ is available.

Studying the Clopper--Pearson interval, \citet{kp1} gave a first-order approximation of $E(L_{CP})$ in the form of beta quantiles and used that to numerically calculate the  sample size required to obtain a desired expected length $d$. Ignoring the higher terms of the expansion (\ref{lenexp}) we obtain the second-order approximation $E(L_{CP})\approx 2z_{\alpha/2}n^{-1/2}(pq)^{1/2}+n^{-1}$, which can be evaluated analytically. Given an initial guess $p_0$ for $p$, the equation
$2z_{\alpha/2}n^{-1/2}(p_0q_0)^{1/2}+n^{-1}=d$
has the solution
\begin{equation}\label{nekv}
n = \Big\lceil\frac{2z^2p_0q_0+2z\sqrt{z^2p_0^2q_0^2+dp_0q_0}+d}{d^2}\Big\rceil
\end{equation}
when rounded up to the nearest integer. This is a good approximation of the actual required sample size, with a small positive bias. At the 95 \% level it does typically not differ by more than 4 from the solution obtained by more complicated (and computer-intensive) exact numerical computations. For $p$ close to 1/2, the Krishnamoorthy--Peng method is slightly more accurate, whereas for $p$ close to 0 or 1, (\ref{nekv}) gives a better approximation. In either case, both approximations are accurate enough for most applications. As an example, when $p_0=0.05$ and $d=0.05$, the actual required sample size is 329, while our approximation yields $n=331$, corresponding to an actual expected length of 0.0498. In comparison with exact methods or the Krishnamoorthy--Peng procedure, (\ref{nekv}) offers greater computational ease without sacrificing much accuracy.

It is likewise possible to solve the cubic equation that results from including the $n^{-3/2}$-term of (\ref{lenexp}), but the solution does not yield a simple formula and does not give substantially improved accuracy.

A downside to this approach to sample size determination is that the initial guess $p_0$ may be quite wrong. This is particularly problematic if $p$ is closer to 1/2 than is $p_0$, in which case the calculated required sample size will be too small. As a safety measure, it is sometimes recommended to use the conservative guess $p_0=1/2$, which maximizes the required sample size. More often than not, however, this choice is needlessly conservative.

An alternative approach, with a Bayesian flavour, is to use a prior distribution for $p$ when determining the sample size. Beta distributions constitute a flexible and analytically tractable class of priors for $p$. For $p\sim Beta(a,b)$, we have
$$E\Big(2z_{\alpha/2}n^{-1/2}(pq)^{1/2}+n^{-1}\Big)=2z_{\alpha/2}n^{-1/2}\frac{\Gamma(a+1/2)\Gamma(b+1/2)}{(a+b)\Gamma(a)\Gamma(b)}+n^{-1}.$$
With $R(a,b)=\Gamma(a+1/2)\Gamma(b+1/2)\lbrack(a+
b)\Gamma(a)\Gamma(b)\rbrack^{-1}$, this gives the required sample size
$$n = \frac{2z_{\alpha/2}^2R^2(a,b)+2z_{\alpha/2}\sqrt{z_{\alpha/2}^2R^4(a,b)+dR^2(a,b)}+d}{d^2}.$$

When applying a frequentist procedure, the prior information about $p$ is typically diffuse, indicating that a low-informative prior should be used so as not to bias the sample size determination. One example is the Jeffreys prior $Beta(1/2,1/2)$, which puts more probability mass close to 0 and 1 and yields $R(1/2,1/2)=1/\pi$. Other examples include the uniform $Beta(1,1)$ prior, for which we have $R(1,1)=\pi/8$ and the $Beta(2,2)$ prior, which puts more mass close to 1/2, yielding $R(2,2)=9\pi/64$.

The required sample size for different combinations of $p$ and $\alpha$ is shown in Figure \ref{nfig2}. It is decreasing in $\alpha$, increasing in $p$ when $p< 0.5$ and decreasing in $p$ when $p> 0.5$.

\begin{figure}[H]
\begin{center}
   \caption{The required sample size for the Clopper--Pearson interval for different combinations of $p$ and $\alpha$.}\label{nfig2}
   \includegraphics[width=\textwidth]{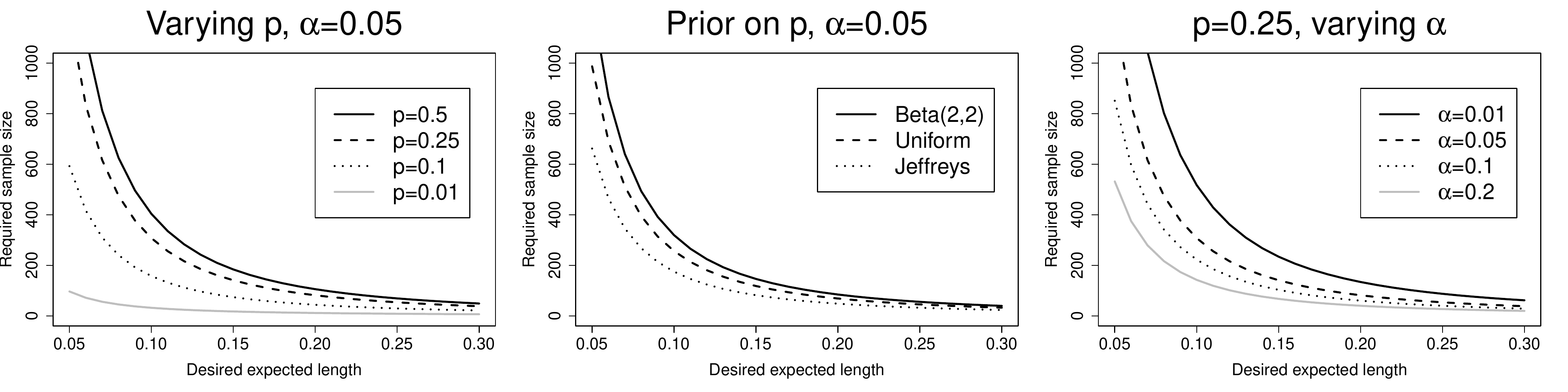}
\end{center}
\end{figure}

\emph{Remark.} In formulas similar to those above, some authors use $d$ to denote the expected half-length, or error tolerance, of a confidence interval. This may be inappropriate in the binomial setting, since using the half-length might give the false impression that all confidence intervals are symmetric about the unbiased estimator $\hat{p}=X/n$. This is not the case for the Clopper--Pearson interval and most good approximate intervals, including those presented in Section \ref{appendix1}. As an example, when $n=50$ and $p=0.01$, the expected length of the Clopper--Pearson interval is 0.044. Since the interval is boundary respecting, most of its length will be placed above $p$. The expected length is very much an interesting quantity when determining sample size, but for binomial proportions it should not be interpreted in terms of error tolerances.


\subsection{The cost of using the exact interval}
Next, we will study the cost of using the exact Clopper--Pearson interval instead of an approximate interval. We will do so by comparing the exact interval to three of the approximate intervals described in Section \ref{appendix1}: the Wilson score, Jeffreys and Agresti--Coull intervals. These intervals have been recommended as default intervals for a single proportion by several authors \citep{ac1,bcd1,ne1}.

First, we measure the cost in terms of increased expected length. By comparing the expansion in Theorem \ref{lenthm} to the expansions in Theorem 7 of \citet{bcd2}, we get the following expressions for how much the expected length of the confidence interval increases when the Clopper--Pearson interval is used instead of an approximate interval.
\begin{corollary}\label{lencor1}
The Clopper--Pearson interval is asymptotically wider than the approximate intervals described in Section \ref{appendix1}. In particular, compared to the length $L_J$ of the Jeffreys interval,
\begin{equation}\label{lengthcomp1}
E(L_{CP})=E(L_J)+n^{-1}+O(n^{-2}),
\end{equation}
\noindent
and if $L_A$ denotes the length of the Wilson or Agresti--Coull interval,
\begin{equation}\label{lengthcomp2}
E(L_{CP})=E(L_A)+n^{-1}+O(n^{-3/2}).
\end{equation}
\end{corollary}
\noindent
Expanded versions of (\ref{lengthcomp2}) for the different intervals, including the $n^{-3/2}$-terms, are given in the proof in the appendix.

Up to $O(n^{-3/2})$, the increase in expected length is inversely proportional to $n$. Note that, up to $O(n^{-3/2})$, the increase does not depend on $p$ or $\alpha$. The cost of using an exact interval, in terms of expected length, is thus more or less constant for a fixed $n$. This is an interesting and somewhat unexpected fact, since the expected lengths of these confidence interval are highly dependent on both $p$ and $\alpha$.\\

\noindent
Next, we consider required sample size. As the Clopper--Pearson interval is wider than the approximate intervals, it naturally requires larger sample sizes to obtain a particular expected length $d$. Let $n_{CP}(d,p,\alpha)$ be the minimum sample size for which $E_p(L_{CP})\leq d$ at the $1-\alpha$ level. Similarly, let $n_{J}(d,p,\alpha)$ be the minimum sample size for which the expected length of the Jeffreys interval is at most $d$ under $p$ at the $1-\alpha$ level.

As noted by \citet{pi1}, the sample size for the Jeffreys interval is well approximated by $n_{J}(d,p_0,\alpha)=4z_{\alpha/2}^2p_0q_0d^{-2}$. Comparing this to (\ref{nekv}) without rounding, the increase in required sample size $n^+_J(d,p_0,\alpha)=n_{CP}(d,p_0,\alpha)-n_{J}(d,p_0,\alpha)$ can be approximated by
\begin{equation}\label{napprox}
n^+_J(d,p_0,\alpha)\approx\frac{d-2z_{\alpha/2}\Big(z_{\alpha/2}p_0q_0
-\sqrt{(z_{\alpha/2}p_0q_0)^2+dp_0q_0}\Big)}{d^2}.
\end{equation}

This approximation is quite accurate, generally differing by less than 1 when compared to the value for $n^+_J$ obtained using substantially more computer-intensive exact computations.

(\ref{napprox}) is plotted as a function of $d$ for three choices of $p_0$ in Figure \ref{nfigWAC}. When shorter intervals are desired, the increase in required sample size can be substantial. When $d=0.05$, for instance, $n^+_J$ is 40 for $0.05\leq p_0\leq 0.95$.

As was the case for the expected length, the increase $n^+_J$ is remarkably insensitive to $p$ and $\alpha$: there is no concernable difference when $0.05\leq p \leq 0.95$ and $0.001\leq\alpha\leq0.2$. The cost of using an exact interval instead of the Jeffreys interval is, in terms of required sample size, constant for a fixed expected length $d$.

\begin{figure}
\begin{center}
   \caption{The increase in required sample size when using the Clopper--Pearson interval instead of the Jeffreys, Wilson score and Agresti--Coull intervals, as approximated by (\ref{napprox})--(\ref{napproxAC}).}\label{nfigWAC}
   \includegraphics[width=\textwidth]{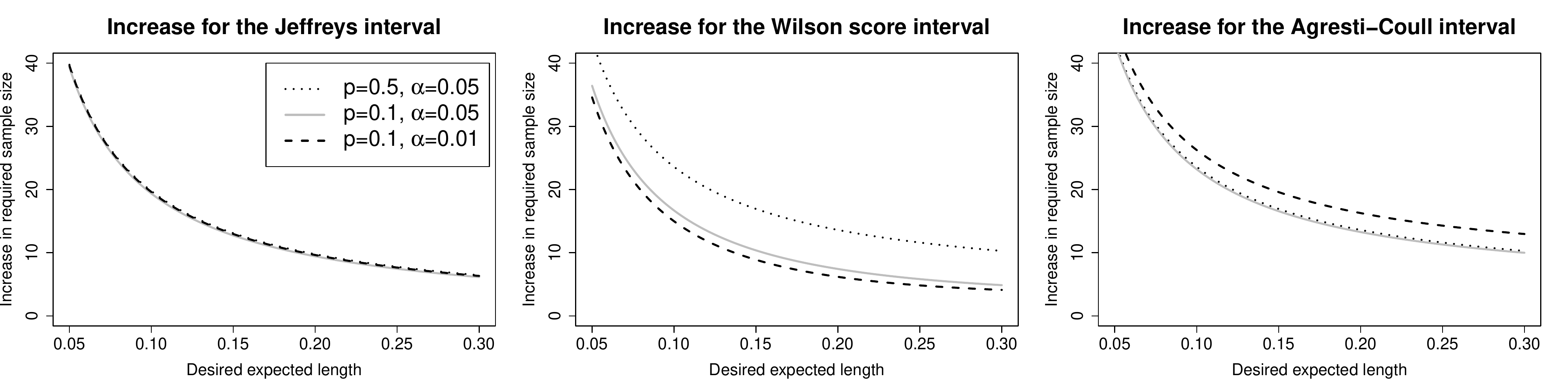}
\end{center}
\end{figure}

Moving on to the Wilson score interval, \citet{pi1} gave the following formula for its sample size: $$n_{WS}(d,p_0,\alpha)=z^2_{\alpha/2}\lbrack p_0q_0+d^2/2+\sqrt{p_0^2q_0^2+d^2(p_0-1/2)^2}\rbrack \lbrack d^2/2\rbrack^{-1}.$$ The increase $n^+_{WS}(d,p_0,\alpha)=n_{CP}(d,p_0,\alpha)-n_{WS}(d,p_0,\alpha)$ can thus be approximated by
\begin{equation}\label{napproxW}\begin{split}
n^+_{WS}(d,p_0,\alpha)\approx d^ {-2}\Big\lbrack d(1+dz_{\alpha/2}^2)+2z_{\alpha/2}\Big(&\sqrt{z_{\alpha/2}^2p_0^2q_0^2+dp_0q_0}\\
&-\sqrt{z_{\alpha/2}^2p_0^2q_0^2+d^2z_{\alpha/2}^2(p_0-1/2)^2}\Big)\Big\rbrack.
\end{split}\end{equation}
The approximation is good when $p_0$ is not very small, typically not differing by more than 2 from the exact value.

Similarly, \citet{pi1} gave the formula $n_{AC}(d,p_0,\alpha)=4z_{\alpha/2}^2p_0q_0d^{-2}-z^2_{\alpha/2}$ for the sample size of the Agresti--Coull interval. Consequently, the increase $n^+_{AC}(d,p_0,\alpha)=n_{CP}(d,p_0,\alpha)-n_{AC}(d,p_0,\alpha)$ is approximately
\begin{equation}\label{napproxAC}\begin{split}
n^+_{AC}(d,p_0,\alpha)\approx \frac{d+z_{\alpha/2}^2(d^2-2p_0q_0)+2z_{\alpha/2}\sqrt{(z_{\alpha/2}p_0q_0)^2+dp_0q_0}}{d^2}.
\end{split}\end{equation}
The expressions (\ref{napproxW}) and (\ref{napproxAC}) are plotted for some combinations of $p$ and $\alpha$ in Figure \ref{nfigWAC}. For the Agresti--Coull interval, the cost is  more or less constant in $p$, but is sensitive to changes in $\alpha$. For the Wilson score interval, the cost depends on both $p$ and $\alpha$.


\subsection{The exact frequentist interval and Bayesian credible intervals with noninformative priors}\label{freqbayes}
Equation (\ref{lengthcomp1}) in Corollary \ref{lencor1} and the fact that (\ref{napprox}) is so insensitive to $p$ and $\alpha$ reveal a strong connection between the frequentist Clopper--Pearson interval and the Bayesian credible interval derived under the Jeffreys prior. In the light of these results, it seems natural to think of the Bayesian interval as a sort of continuity-correction of the Clopper--Pearson interval, in which conservativeness is sacrificed in order to get a short interval.

Attempts to connect the exact frequentist interval with Bayesian intervals have previously been made by \citet{bcd1}, who argued that the Jeffreys interval can be thought of as a continuity-corrected version of the Clopper--Pearson interval. Their argument comes from a comparison between the Jeffreys interval and the mid-p interval, which generally is considered to be a continuity-corrected Clopper--Pearson interval. However, the key step in their argument is their equation (17), which is incorrect; it relies on the false assumption that for two continuous functions $f_1$ and $f_2$, $(f_1+f_2)^{-1}=f_1^{-1}+f_2^{-1}$.

Another natural noninformative Bayesian interval is that based on the uniform prior, $Beta(1,1)$. The Clopper--Pearson interval is essentially this interval \emph{after the prior information has been removed}, a fact which we have not seen mentioned before in the literature. To see this, note that for a central Bayesian interval with prior $Beta(a,b)$, $a,b>0$, the lower bound is given by the beta quantile $p_{L,B}(a,b,X,n)=B(\alpha/2,X+a,n-X+b)$. The parameters $a$ and $b$ can be interpreted as additional successes and failures added to the data. For the uniform prior, $a=b=1$. The lower bound of the Clopper--Pearson interval is similarly the beta quantile $B(\alpha/2,X,n-X+1)$. When $X\notin\{0,n\}$ this can be written as $B(\alpha/2,(X-1)+1,(n-1)-(X-1)+1)=p_{L,B}(1,1,X-1,n-1)$, the lower bound of the $Beta(1,1)$ interval with one success and one failure removed. Expressed in words, the lower bound of the Clopper--Pearson interval equals the lower bound of the Bayesian interval with the uniform prior after the prior information has been removed. Similarly, the upper bound is $1-p_{L,B}(a,b,n-X,n-1)$, i.e. 1 minus the lower bound for $q$ under the uniform prior with one success and one failure removed. The $Beta(1,1)$ interval can thus be thought of as a shrinkage Clopper--Pearson interval.

%

\section{One-sided bounds}\label{onesided}


\subsection{Expected distance to the true proportion}

For one-sided confidence bounds, it is not the expected length that is of interest, but how close the bound is to $p$. Let $L_{U,CP}=p_U-p$ denote the distance from $p_U$ to $p$. The next theorem gives an asymptotic expansion for the expectation of $L_{U,CP}$.

\begin{theorem}\label{lenthm2}
As $n\rightarrow\infty$ the expected distance to $p$ for the $1-\alpha$ one-sided Clopper--Pearson upper bound is
\begin{equation}\label{lenexp2}\begin{split}
E(L_{U,CP})=&n^{-1/2}z_\alpha(pq)^{1/2}+(3n)^{-1}\Big{(}2(1/2-p)z_{\alpha}^2+1+q\Big{)}\\
&+n^{-3/2}z_{\alpha}(pq)^{1/2}\Big{(}-\frac{53}{36}+\frac{\frac{1}{2}-p}{q}+\frac{z_{\alpha}^2+\frac{13}{2}}{36pq}-\frac{13z_{\alpha}^2}{36}\Big{)}+O(n^{-2})
\end{split}\end{equation}
\end{theorem}

The expansion (\ref{lenexp2}) is compared to the actual expected distance to $p$ in Figure \ref{len2}. Like the expansion for the expected length of the two-sided interval, (\ref{lenexp2}) is close to the actual expected distance even for small $n$.

\begin{figure}[H]
\begin{center}
   \caption{Comparison between the actual expected distance and the expansion (\ref{lenexp2}) for the nominal 95 \% Clopper--Pearson upper bound.}\label{len2}
   \includegraphics[width=\textwidth]{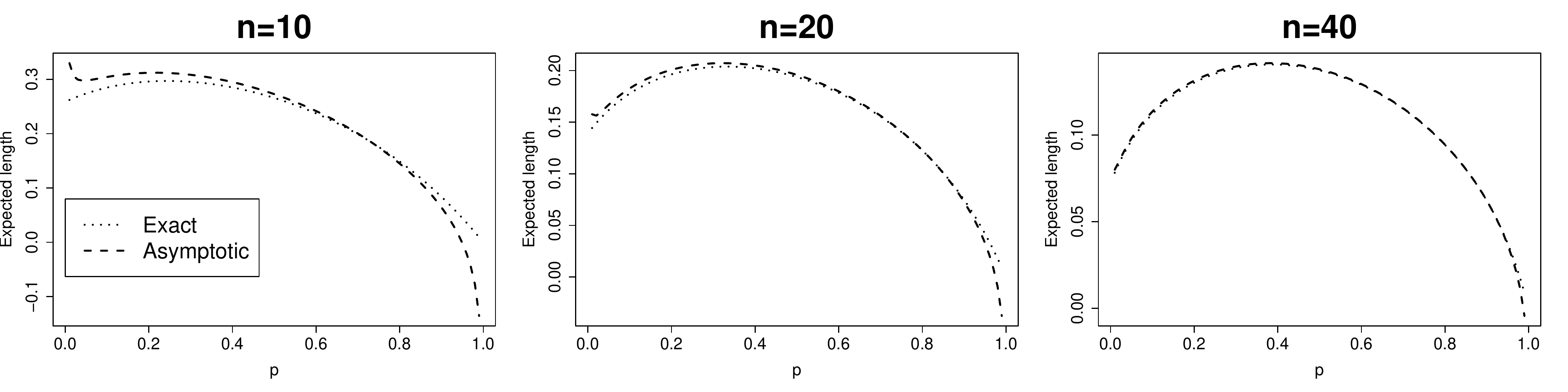}
\end{center}
\end{figure}


\subsection{Sample size determination}
The expressions we obtain in the one-sided case are not quite as simple as those in the two-sided case. Let $d$ denote the desired expected distance to $p$ and let $p_0$ be the initial guess for the value of $p$. Proceeding as before, using the second-order approximation $$E(L_{U,CP})\approx n^{-1/2}z_\alpha(pq)^{1/2}+(3n)^{-1}\Big{(}2(1/2-p)z_{\alpha}^2+1+q\Big{)}$$ yields the required sample size
\[\begin{split}
\Big\lceil n=(2d^ 2)^{-1}\Big(9z_\alpha^2p_0q_0&+3z_\alpha\sqrt{3p_0q_0}\sqrt{3z_\alpha^2p_0q_0+4\lbrack dz_\alpha^2-2dz_\alpha^2p_0+d(1+q_0)\rbrack}\\&\qquad\qquad\qquad\qquad\qquad~+6\lbrack 2z_\alpha^2(1/2-p_0)+(1+q_0)\rbrack\Big)\Big\rceil.
\end{split}\]

This approximation is very good when $d$ is not too small. For smaller $d$ it has a small negative bias: when $\alpha=0.05$ and $p_0=1/2$ the actual required sample size for $d=0.02$ is $n=1738$, whereas the above expression gives the approximation $n=1721$, corresponding to a true expected distance of $d=0.0201$. For most purposes, this will probably be a sufficiently accurate approximation.

As in the two-sided case, we may consider using a prior distribution of $p$, rather than a fixed $p_0$, to determine a reasonable sample size. The expectation of the second-order approximation with respect to a $Beta(a,b)$ prior for $p$ is
\begin{equation}\label{bayesexp}
\frac{(2+z_\alpha^2)\Gamma(2-a)\Gamma(2-b)}{3n\Gamma(4-a-b)}-\frac{(2z_\alpha^2+1)\Gamma(3-a)\Gamma(2-b)}{3n\Gamma(5-a-b)}+\frac{z_\alpha\Gamma(5/2-a)\Gamma(5/2-b)}{\sqrt{n}\Gamma(5-a-b)}.
\end{equation}
Note that this expression is undefinied when $a,b\geq 2$, limiting which priors we can use. When (\ref{bayesexp}) is well-defined, a general formula for the required sample size can be obtained by equating (\ref{bayesexp}) to $d$ and solving for $n$, but the resulting expression is rather complicated. It is however readily evaluated for particular values of $a$ and $b$. For the Jeffreys prior for instance, the required sample size is
$$\Big\lceil n=\frac{6z_\alpha(z_\alpha+\sqrt{z_\alpha^2+9d\pi})}{d^2}+\frac{\pi}{16d}\Big\rceil.$$
The solutions for the Jeffreys and uniform priors as well as the low-informative asymmetric $Beta(1/2,1)$ prior are shown in Figure \ref{nfig3}, along with the solutions for fixed $p_0$ and different values of $\alpha$.

In contrast to the two-sided case, $d$ can in fact be interpreted as an error tolerance for the one-sided bound. This makes the interpretation of $d$ easier in this case.

\begin{figure}[H]
\begin{center}
   \caption{The required sample size for the upper Clopper--Pearson bound for different combinations of $p$ and $\alpha$.}\label{nfig3}
   \includegraphics[width=\textwidth]{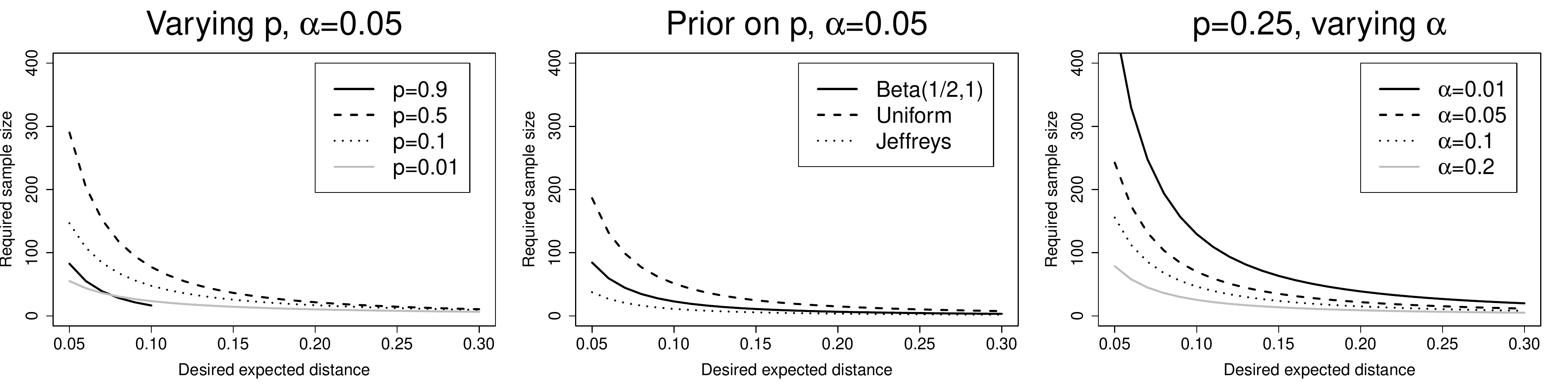}
\end{center}
\end{figure}


\subsection{The cost of using the exact bound}
The cost of using the exact bound will be evaluated in relation to three approximate bounds: The Jeffreys, second-order correct and modified loglikelihood root bounds, described in Section \ref{appendix1}. Comparing (\ref{lenexp2}) to the expansions in Corollary 1 of \citet{cai1} and Proposition 2.2 of \citet{st1}, the following corollary is immediate.

\begin{corollary}
When $L_{U,A}$ denotes the distance of the Jeffreys, second-order correct or modified loglikelihood root bounds,
$$E(L_{U,CP})=E(L_{U,A})+(2n)^{-1}+O(n^{-3/2}).$$
\end{corollary}
It should be noted that there are one-sided versions of the Wald and Wilson score intervals, but since these have very poor performance \citep{cai1} they are omitted from our comparison. They can however readily be compared to the Clopper--Pearson bound by comparing (\ref{lenexp2}) to the corresponding expansions in Corollary 1 of \citet{cai1}.

For one-sided bounds, the approximation of the increased sample size when the exact bound is used is more involved than it was for the two-sided cases. To preserve space, we simply use the naive first-order formula $n=z^2_{\alpha/2}pqd^{-2}$ to determine the sample sizes for the approximate bounds. This works reasonably well most of the time. Let $n^+(d,p,\alpha)$ be the increase in sample size when the Clopper--Pearson bound is used instead of an approximate bound. Then, with $\omega(z,d,p)=9z^2pq+12dz^2-24dz^2p$,
\begin{equation}\label{napprox2}\begin{split}
n^+(d,p_0,\alpha)\approx\frac{\sqrt{\omega(z_{\alpha},d,p_0)+12d(1+q_0)}-
\sqrt{\omega(z_{\alpha},d,p_0)+12d(1/2-p_0)}+\frac{d}{2}}{d^{2}}.
\end{split}\end{equation}
Compared to the increased sample size in the two-sided setting, (\ref{napprox2}) is more sensitive to changes in $p$ and $\alpha$. The cost is the smallest when $p=0.5$. When evaluating the increased sample size $p_0=0.5$ is therefore not to be recommended as the default choice, as this can lead to a serious underestimation of the increase, especially for smaller $d$.

\begin{figure}[H]
\begin{center}
   \caption{The increase in required sample size when using the upper Clopper--Pearson bound instead of an approximate upper bound, as approximated by (\ref{napprox2}) for $\alpha=0.05$.}\label{nfig}
   \includegraphics[width=\textwidth]{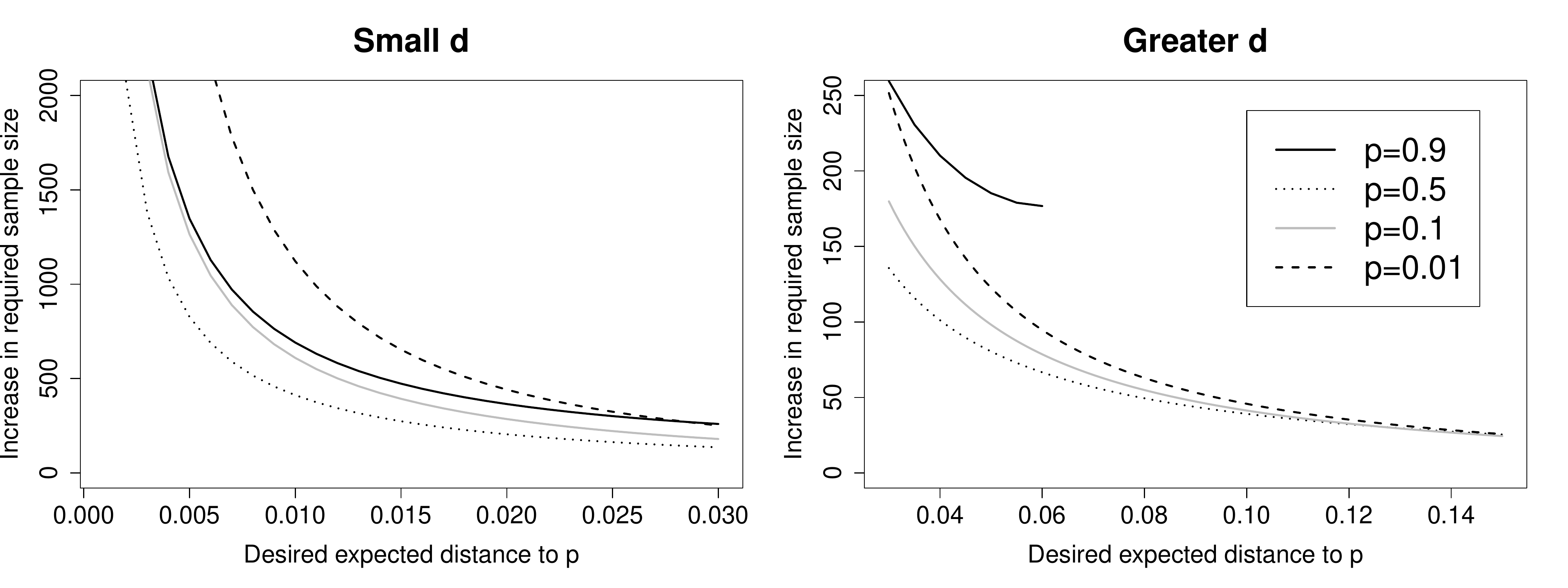}
\end{center}
\end{figure}

\section{Discussion}\label{disc}

\subsection{Minimum coverage or mean coverage?}
The Clopper--Pearson methods are exact in the sense that their minimum coverage over all $p$ is at least $1-\alpha$. An alternative measure of coverage is mean coverage, which typically is taken to be the expected coverage with respect to a uniform pseudo-prior of $p$. In recent papers on binomial confidence intervals, approximate methods have often been considered to be preferable to exact methods \citep{ac1,bcd1,cai1,nn1}, the argument being that it makes more sense to interpret the confidence level as the mean coverage probability rather than the minimum coverage probability, as this corresponds better to how many modern-day statisticians think of coverage levels. Reasoning along the lines of \citet{nn1}, the minimum coverage can occur in an uninteresting part of the parameter space, typically close to the boundaries, possibly rendering it an uninteresting measure of coverage. This is discussed further in the next section.

As noted e.g. by \citet{nn1}, using mean coverage is very much in line with current statistical practice in other problems. Widely used methods based on boostrapping and MCMC, for instance, typically only control confidence levels and type I error rates approximately, attaining the $1-\alpha$ level only on average. This is particularly reasonable when the model is known to be an imperfect representation of the underlying process, in which case even minimum coverage criterions are approximate at best. Unlike in many other applications however, one can often be rather certain that a random variable truly is binomial. This begs the question whether one should resort to approximations or use methods that really are guaranteed to be exact.

If the Bayesian credible intervals based on either the Jeffreys $Beta(1/2,1/2)$ or the uniform $Beta(1,1)$ priors are used, an additional argument for the mean coverage criterion is given by the Bayesian interpretation of these intervals. If we accept mean coverage as a criterion when choosing between confidence intervals, we can obtain intervals that simultaneously admit both frequentist and objective Bayesian interpretations.

The minimum coverage criterion underlying the Clopper--Pearson interval is in line with classical statistical theory. It asserts that overcoverage is a less serious problem than undercoverage, or, in other words, that it is better to be more confident than you think that you are than to be overconfident. Next, in order to evaluate this argument further, we will discuss just how overconfident one risks being when using approximate intervals.


\subsection{The cost of using approximate methods}
Just as there are costs associated with using exact methods, there are costs associated with using approximate methods: the actual coverage level may, even for large $n$, drop below the nominal $1-\alpha$. There is no guarantee that the true $p$ is not in an unfortunate area with low coverage. However, these coverage anomalies usually occur close to the boundaries of the parameter space, so unless we are interested in inference for $p$ close to 0 or 1, it may therefore be more relevant to investigate the minimum over a central subset, such as $\lbrack 0.1,0.9\rbrack$.

The problem of undercoverage is illustrated in Figure \ref{acost}, in which the minimum coverages of the Jeffreys, Wilson and Agresti--Coull intervals are shown for different $n$ when the minimum is taken over either $p\in\lbrack0.01,0.99\rbrack$ or $p\in\lbrack0.1,0.9\rbrack$. For $p\in\lbrack0.01,0.99\rbrack$ and a moderately large sample size of $n=250$, the minimum coverage of the Jeffreys interval is approximately 0.88, whereas the minimum coverage of the Wilson score interval is about 0.93. The Agresti--Coull interval fares somewhat better, with a minimum coverage of 0.94. In this setting neither the Jeffreys nor the Wilson score interval has a minimum coverage above 0.94 even for a sample size as large as $n=2000$.

\begin{figure}
\begin{center}
   \caption{Minimum coverage of two-sided approximate intervals over $p\in\lbrack 0.01,0.99\rbrack$ or $p\in\lbrack 0.1,0.9\rbrack$ when $\alpha=0.05$, computed over a grid of 200,000 equidistant points.}\label{acost}
   \includegraphics[width=\textwidth]{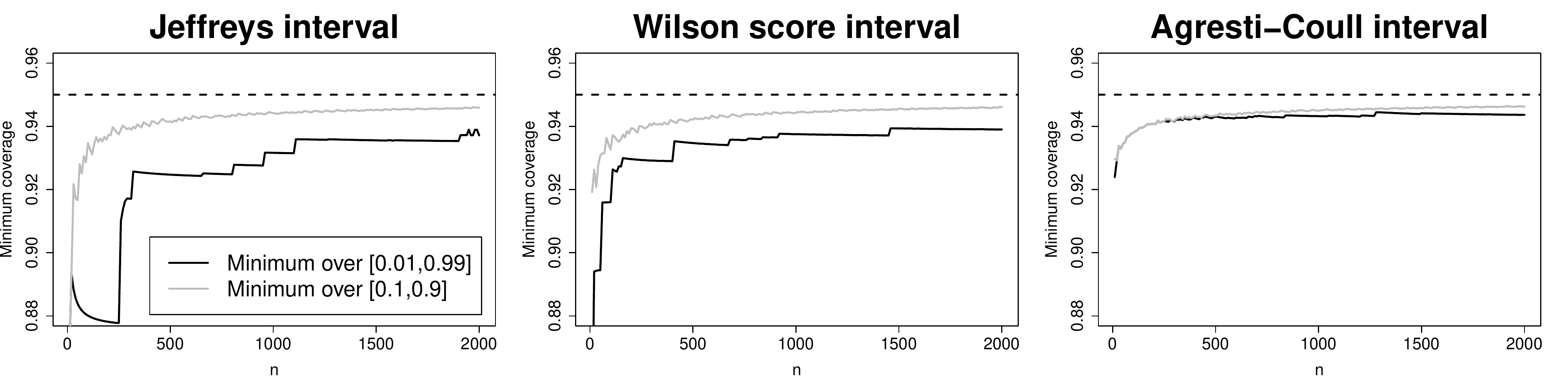}
\end{center}
\end{figure}

A coverage of 0.94 for a nominal 0.95 method is well below what one should expect for sample sizes as large as $n=2000$. If undercoverage of this size is unacceptable, one may apply computer-intensive coverage-adjustment method similar to those discussed in \citet{re1}, decreasing $\alpha$ to some $\gamma$ for which the minimum coverage over some set of values of $p$ is at least $1-\alpha$, thus making the methods exact. Decreasing $\alpha$ will however \emph{increase} the expected length of the intervals. 

Comparing sample sizes of the $1-\gamma$ Jeffreys interval and the $1-\alpha$ Clopper--Pearson interval, we have:
\[
n^+(d,p_0,\alpha,\gamma))\approx\frac{d+2p_0q_0(z_{\alpha/2}^2-2z_{\gamma/2}^2)
+2z_{\alpha/2}\sqrt{z_{\alpha/2}^2p_0^2q_0^2+dp_0q_0}}{d^2}.
\]
For $n$ between 1000 and 1500, computer-intensive adjustments of the Jeffreys interval lead to $\gamma\approx 0.04$ (the actual $\gamma$ being somewhat larger than $0.04$). For $p_0=1/2$ and $d=0.04$, we get $n^+(0.04,1/2,0.05,0.04))\approx -186$, i.e. that the Clopper--Pearson interval requires 186 observations \emph{fewer} to obtain the desired expected length. In general, approximate intervals that have been adjusted to be exact are outperformed by the Clopper--Pearson interval.

Similarly, if one is willing to use approximate intervals, it is possible to apply coverage-adjustments to the Clopper--Pearson interval in order to adjust its mean coverage to $1-\alpha$. The resulting $\gamma>\alpha$, meaning that the interval becomes shorter after the adjustment. \citet{thu1} studied this problem in detail for $n\leq 100$, showing that the adjusted Clopper--Pearson intervals often outperformed its competitors.

%

It should be noted that other criterions than coverage and expected length can be used for comparing confidence intervals. \citet{ne2,ne1} compared location properties, i.e. left and right non-coverage, of intervals and found the Clopper--Pearson interval to have good properties in comparison to some approximate intervals. \citet{vh2} considered two criterions related to $p$-values, motivated by the interpretation of confidence intervals as inverted tests, and found the Clopper--Pearson interval to be better than its competitors.


\subsection{Conclusions}
When choosing between exact and approximate confidence methods, it is important to be aware of the benefits and the costs associated with the two types of methods. The coverage fluctuations of approximate intervals have been compared in several studies, making it easy for practitioners to compare how costly these intervals can be in terms of undercoverage. We have attempted to make the costs of using exact methods explicit, by giving expressions for how much larger the expected length of the exact intervals are and for how much the sample size increases when a fixed expected length is to be attained.

For the two-sided Jeffreys interval, exactness comes at a fixed price: the cost of using the Clopper--Pearson interval instead of this intervals is, in terms of expected length and required sample size, insensitive to $p$ and $\alpha$. For the Agresti--Coull interval, the cost only depends on $\alpha$. This stands in contrast to the Wilson score interval and one-sided bounds, for which $p$ and $\alpha$ can greatly affect the cost. In either case the required sample sizes for the exact methods can be substantially larger than those of the approximate methods.

In our comparison of exact and approximate methods, the only exact methods considered were the Clopper--Pearson interval and bound. While other shorter exact two-sided intervals exist, they suffer from various problems that make them unsuitable for use. Moreover, the Clopper--Pearson interval is used far more often than the other exact intervals, which merits its role as the main subject of this study.

\subsubsection*{Acknowledgement}
The author would like to thank Robert Newcombe and Silvelyn Zwanzig for their many thoughtful comments on an earlier version of this paper.

\appendix

\section*{Appendix: Proofs}\label{appendix2}

Theorem \ref{boundapprox1} follow directly from the following lemma, which is used in the proofs of Theorems \ref{lenthm} and \ref{lenthm2}.

\begin{lemma}\label{limlem}
With assumptions and notation as in Theorem \ref{boundapprox1}, the bounds of the Clopper--Pearson interval are
\[\begin{split}
p_L=\hat{p}&-n^{-1/2}z_{\alpha/2}(\hat{p}\hat{q})^{1/2}+(3n)^{-1}\Big{(}2(1/2-\hat{p})z_{\alpha/2}^2-(1+\hat{p})\Big{)}\\
&-n^{-3/2}z_{\alpha/2}(\hat{p}\hat{q})^{1/2}\Big{(}-\frac{53}{36}-\frac{\frac{1}{2}-\hat{p}}{\hat{p}}+\frac{z_{\alpha/2}^2+11}{36\hat{p}\hat{q}}-\frac{13z_{\alpha/2}^2}{36}\Big{)}+O(n^{-2}),\\
p_U=\hat{p}&+n^{-1/2}z_{\alpha/2}(\hat{p}\hat{q})^{1/2}+(3n)^{-1}\Big{(}2(1/2-\hat{p})z_{\alpha/2}^2+(1+\hat{q})\Big{)}\\
&+n^{-3/2}z_{\alpha/2}(\hat{p}\hat{q})^{1/2}\Big{(}-\frac{53}{36}+\frac{\frac{1}{2}-\hat{p}}{\hat{q}}+\frac{z_{\alpha/2}^2+11}{36\hat{p}\hat{q}}-\frac{13z_{\alpha/2}^2}{36}\Big{)}+O(n^{-2}).
\end{split}\]
\end{lemma}
The approximations are close to the actual bounds even for small sample sizes. When $n=25$ and $\hat{p}$ is not too close to $0$ or $1$, the approximations are typically accurate up to at least least two decimal places.

\begin{proof}[Proof of Lemma \ref{limlem}]
First, we note that the lower limit of the Bayesian interval with prior $Beta(a,b)$, $a,b>0$, is given by the beta quantile $p_B(a,b,X,n)=B(\alpha/2,X+a,n-X+b)$. 

For the Clopper--Pearson interval $p_L$ is the beta quantile $B(\alpha/2,X,n-X+1)$. When $X\notin\{0,n\}$ this can be written as $B(\alpha/2,(X-1)+1,(n-1)-(X-1)+1)$, i.e. $p_B(1,1,X-1,n-1)$, the lower limit of the $Beta(1,1)$ interval for $X-1$ and $n-1$. 

An asymptotic expression for $p_B(a,b,X,n)$ in terms of $\tilde{p}=(X+a-1)/(n+a+b-2)$ is given in expression (A.23) in \citet{bcd2}. We obtain the asymptotic expansion of $p_L$ by taking the expansion for $p_B(1,1,X-1,n-1)$ and rewriting the bound in terms of $X/n$, in a manner similar to equation (A.26) of \citet{bcd2}. The expansion of $p_U$ is derived analogously.
\end{proof}


\begin{proof}[Proof of Theorem \ref{lenthm}]
Using the expansion in Lemma \ref{limlem}, when $X\notin\{0,n\}$
\[\begin{split}
L_{CP}=p_U-p_L=&2n^{-1/2}z_{\alpha/2}(\hat{p}\hat{q})^{1/2}+n^{-1}+n^{-3/2}m(\hat{p})+R_n,
\end{split}\]
where
\[
m(\hat{p})=(\hat{p}\hat{q})^{-1/2}\frac{z_{\alpha/2}}{18}\Big{(}z_{\alpha/2}^2+2-17\hat{p}\hat{q}-13\hat{p}\hat{q}z_{\alpha/2}^2\Big{)}
\]
and $E(R_n)=O(n^{-2})$ by Theorem 7 of \citet{bcd2}. As the contribution to expected length given by $X\in\{0,n\}$ is $P(X\in\{0,n\})\cdot(1-(\alpha/2)^{1/n})=O((1/2)^{n})$, when computing $E(L_{CP})$ we can disregard the fact that the above expansion is invalid for $X\in\{0,n\}$.

The $n^{-1/2}$ term is the length of the Wald interval, the expectation of which was given in \citet{bcd2}:
\[
E\Big{(}2z_{\alpha/2}n^{-1/2}(\hat{p}\hat{q})^{1/2}\Big{)}=2z_{\alpha/2}n^{-1/2}(pq)^{1/2}\Big{(}1-(8npq)^{-1}\Big{)}+O(n^{-2}).
\]
$m(\hat{p})$ is bounded when $X\neq \{0,n\}$ and $m(p)$ is twice differentiable for $0<p<1$. Thus, by the theorem in Section 27.7 of \citet{cr1},
\[
E(m(\hat{p}))=(pq)^{-1/2}\frac{z_{\alpha/2}}{18}\Big{(}z_{\alpha/2}^2+2-17pq-13pqz_{\alpha/2}^2\Big{)}+O(n^{-1})
\]
and (\ref{lenexp}) follows after all terms of the same order are collected.
\end{proof}


\begin{proof}[Proof of Corollary \ref{lencor1}]
(\ref{lengthcomp1}) and (\ref{lengthcomp2}) are obtained by comparing (\ref{lenexp}) to the expansions in Theorem 7 och \citet{bcd2}. In particular, 
compared to the length $L_{WS}$ of the Wilson score interval,
\[\begin{split}
E(L_{CP})=&E(L_{WS})+n^{-1}\\&-n^{-3/2}\frac{z}{36(pq)^{1/2}}\Big\lbrack 9z\Big(z+\Big(\frac{26}{9}pq-\frac{2}{9}\Big)^2\Big)+34pq(1-2z^2)-4\Big\rbrack+O(n^{-2}),
\end{split}\]
\noindent
compared to the length $L_{AC}$ of the Agresti--Coull interval,
\[\begin{split}
E(L_{CP})=&E(L_{AC})+n^{-1}\\&-n^{-3/2}\frac{z}{36(pq)^{1/2}}\Big\lbrack 9z\Big(2z+\Big(\frac{26}{9}pq-\frac{2}{9}\Big)^2\Big)+pq(34-108z^2)-4\Big\rbrack+O(n^{-2})
\end{split}\]
\end{proof}


The proof of Theorem \ref{lenthm2} is in analogue with the proof of Theorem \ref{lenthm} and is therefore omitted. It relies on the expansion for the expected distance of the one-sided Wald bound found in Corollary 1 of \citet{cai1}.

\pagestyle{plain}

~\\[3mm]
Contact information:\\
M{\aa}ns Thulin, Department of Mathematics, Uppsala University, Box 480, 751 06 Uppsala, Sweden\\
E-mail: thulin@math.uu.se
\end{document}